\documentclass[12pt, a4paper]{amsart}

\usepackage[utf8]{inputenc}
\usepackage{comment}
\usepackage{hyperref}
\usepackage{amsthm}
\usepackage{amsfonts,amssymb,tikz,mathrsfs}
\usepackage{amsmath}
\hypersetup{
     colorlinks=true,
     linkcolor=blue,
     citecolor=blue}
\usepackage{tikz}
\usepackage{dsfont}

\newtheorem{theorem}{Theorem}[section]
\newtheorem{corollary}[theorem]{Corollary}
\newtheorem{lemma}[theorem]{Lemma}
\newtheorem{Example}[theorem]{Example}
\newtheorem{proposition}[theorem]{Proposition}
\numberwithin{equation}{section}

\newtheorem{remark}[theorem]{Remark}

\begin{document}

 \title{Positive Self-Dual Hopf Algebras of Galois characters}
\author{Farid Aliniaeifard }
\address[Farid Aliniaeifard]
{Department of Mathematics \\University of Colorado {\bf Boulder}\\
Boulder, 80309\\
USA}
\email{farid.aliniaeifard@colorado.edu}
\author{Shawn Burkett}
\address[Shawn Burkett]
{Department of Mathematics \\University of Colorado {\bf Boulder}\\
	Boulder, 80309\\
	USA}
\email{shawn.burkett@colorado.edu }

\maketitle

\begin{abstract}
{
By using the action of certain Galois groups on complex irreducible characters and conjugacy classes, we define the Galois characters and Galois classes.
We will introduce a set of Galois characters, called Galois irreducible characters, and we show that each Galois character is a non-negative linear combination of the Galois irreducible characters. It is shown that whenever the complex characters of the groups of a tower produce a positive self-dual Hopf algebra (PSH), Galois characters of the groups of the tower also produce a PSH. Then we will classify the Galois characters and Galois classes of the general linear groups over finite fields. In the end, we will precisely indicate the isomorphism between the PSH of Galois characters and a certain tensor product of positive self-dual Hopf algebras.	\vspace{3mm}
 }
\end{abstract}


\section[Introduction]{Introduction}
Character theory is at the heart of mathematics with many applications in other sciences such as physics and chemistry. Given a finite group $G$, fix a positive integer $d$ that divides $|G|$.
Galois characters of the group $G$ are those characters with the image in $\mathbb{Q}[\zeta_d]$, where $\zeta_d$ is a $d$th primitive root of unity. It seems that Galois characters are playing a significant role in the application of character theory.  All characters of any group are Galois where $d$ is the exponent of $G $, the least common multiple of the orders of all elements of $G$. The character theory of the
symmetric groups and symmetric function theory are identical. This connection is beautifully presented in
Macdonald’s book \cite{MC}. Geissinger and Zelevinsky independently accomplished that these two theories have Hopf structures which illuminated the identification \cite{Ge,ZE}. The supercharacters of the supercharacter theory in \cite{DI08} for unipotent-upper triangular matrices over a finite field $\mathbb{F}_{p^n}$ are some Galois characters where  $d=p$.
In \cite{AABB12,BT} the authors constructed a Hopf algebra of supercharacters and identified the Hopf algebra of supercharacters of unipotent upper-triangular matrices with the Hopf algebra of symmetric functions in noncommutating variables. We will  construct some positive self-dual Hopf algebras that are identical to the Hopf algebra of Galois characters of tower of general linear groups over a finite field $\mathbb{F}_q$.\\

A character $\chi$ is a {\it $d$-Galois character} of $G$ if $\chi$ is a character of $G$ and $\chi(g)\in \mathbb{Q}[\zeta_d]$ for every $g\in G$. Note that when $\mathrm{exp}(G)\mathop{|}d$, the set of $d$-Galois characters is the same as the set of all characters, and when $d=1$ the set of $1$-Galois characters is the same as the set of integral characters, the characters with image in $\mathbb{Q}$.
Every $d$-Galois character is a non-negative linear combination of a special set of characters called $d$-Galois irreducible characters (see Theorem \ref{GCI}).
If an irreducible character $\psi$ is a constituent of a $d$-Galois irreducible character $\chi$, then there is a constant $m_d(\psi)$ such that  $$m_d(\psi)\chi$$ is an irreducible character corresponding to an irreducible  module over $\mathbb{Q}[\zeta_d]$.  The constant $m_{d}(\psi)$ is called the Schur index of $\psi$ over $\mathbb{Q}[\zeta_d]$. For more details see \cite[Sections 9 and 10]{IS}. Let $p$ be a prime and $d$ be a power of $p$. When $G$ is a $p$-group, $p>2$, any $d$-Galois character is afforded by a  $\mathbb{Q}[\zeta_d]$-representation of $G$ (See \cite[Corollary (10.14)]{IS}). \\

 We will show that whenever the space $\bigoplus_{n\geq 0}\mathsf{cf}(G_n)$ of class functions of a tower of groups $(G_0\leq G_1\leq G_2 \leq G_3 \leq \ldots)$ has a PSH structure, then the Galois class functions of the groups of the tower also have a PSH structure. We will study the Galois characters of some families of groups, especially the general linear groups over a finite field $\mathbb{F}_q$. After classifying the Galois characters and Galois classes, we will find the cuspidal Galois characters, which precisely give us the isomorphism from the PSH of Galois class functions  of the tower of general linear groups to the tensor product of the certain positive self-dual Hopf algebras. 

In section \ref{supercharacter theory}, we give a definition of the supercharacter theory. In section \ref{Galois character theory}, we introduce the Galois character theory. Section \ref{Hopf algebra} will be devoted to the Hopf structures that Galois characters produce, and we will investigate the properties that the Hopf algebra of Galois characters inherited from the Hopf algebra of complex characters. In the last section, we will parametrize the Galois characters and Galois classes of the general linear groups over a finite field $\mathbb{F}_q$, and we identify the PSH structure of Galois characters of general linear groups.

\section{Supercharacter theory}\label{supercharacter theory}

We reproduce the definition of supercharacter theory in \cite{DI08}.

A {\it supercharacter theory} of a finite group $G$ is a pair $(\mathcal{X},\mathcal{K})$, where $\mathcal{K}$ is a partition of $G$ and $\mathcal{X}$ is a partition of ${\rm Irr}(G)$ 
such that:
\begin{enumerate}
	\item The set $\{1\}$ is a member of $\mathcal{K}$.
	\item $|\mathcal{X}|=|\mathcal{K}|$.
	\item The characters $\sigma_X$ are constant on the parts of $\mathcal{K}$.
\end{enumerate}
We refer to the characters $\sigma_X$ as {\it supercharacters} and to the members of $\mathcal{K}$ as {\it superclasses}.

Every finite group $G$ has two trivial supercharacter theories: the usual irreducible character theory
and the supercharacter theory $(\{\{\mathds{1}\},\mathsf{Irr}(G)\setminus\{\mathds{1}\}\},\{ \{1\}, G \setminus \{1\}\} )$, where
$\mathds{1}$ is the trivial character of $G$. However, classifying the set of all  non-trivial supercharacter theories for most groups is still an open problem.

\section{Galois Character Theory }\label{Galois character theory}
Let $d$ be a positive integer such that $d\mathop{|}|G|$ and let $\zeta_d$ be a primitive $d$th root of unity. As we said a $d$-Galois character $\chi$ for a finite group $G$ is a character whose image is in $\mathbb{Q}[\zeta_d]\subseteq \mathbb{C}$.
Let $\mathsf{Irr}(G)$ be the set of irreducible characters of $G$. Irreducible characters  are building blocks of complex characters, and each complex character can be uniquely  written as a non-negative linear combination of irreducible characters, that is $\chi=m_1\psi_1+\ldots+m_k\psi_k$ where each $m_i$ is a positive integer, and each $\psi_i$ is an irreducible character. In this section, we introduce a set of $d$-Galois characters, called  $d$-Galois irreducible characters,  that each $d$-Galois character can be uniquely  written as a non-negative linear combination of $d$-Galois irreducible characters.
\\

Let $G$ be a group of order $n$ and $d\mathop{|}n$, throughout this paper let 
$$\mathsf{Gal}(n,d):=\mathsf{Gal}_{\mathbb{Q}[\zeta_d]} {\mathbb{Q}[\zeta_n]},$$ the Galois group of ${\mathbb{Q}[\zeta_n]}$ over ${\mathbb{Q}[\zeta_d]}$.  There is an action of $\mathsf{ Gal}{(n,d)}$ on $\mathsf{ Irr}(G)$ defined by $\sigma.\chi=\sigma(\chi)$ for every $\sigma\in \mathsf{ Gal}{(n,d)}$ and $\chi\in \mathsf{ Irr}(G)$.  We denote by $\mathsf{O}_{d}(\psi)$ the orbit of an irreducible character $\psi\in \mathsf{Irr}(G)$.  Also, there is a compatible action on the conjugacy classes of $G$, as follows. For every $\sigma\in \mathsf{Gal}{(n,d)}$, there is a unique positive integer $r<n$ coprime to $n$ such that $\sigma(\zeta_n)=\zeta_n^r$, and we let $\sigma.g=g^r$. Let $\mathsf{ GIrr}_d(G)$ be the set of  $\mathsf{ Gal}{(n,d)}$-orbit sums on $\mathsf{ Irr}(G)$, i.e., 
 $$\mathsf{ GIrr}_d(G)=\left\lbrace  \sum_{\phi\in \mathsf{O}_{d}(\psi) }\phi: \psi\in \mathsf{ Irr}(G) \right\rbrace $$ and let $\mathsf{GCl}_d(G)$ be the set of unions of the $\mathsf{ Gal}{(n,d)}$-orbits on the conjugacy classes of $G$. Then by \cite{DI08}, $$( \{ \mathsf{O}_{d}(\chi): \chi \in \mathsf{Irr}(G)  \}  ,\mathsf{GCl}_d(G))$$ is a supercharacter theory. We call the elements of $\mathsf{ GIrr}_d(G)$ the {\it $d$-Galois irreducible characters } and the elements of $\mathsf{GCl}_d$ the {\it $d$-Galois classes} of $G$. Note that if two irreducible characters $\chi$ and $\psi$ are the constituents of the same supercharacter, they have the same degree since $\chi(1)=\sigma(\psi(1))=\psi(1)$ for every element $\sigma$ in the Galois group. Remark that when we have a smaller Galois group, we have a finer supercharacter theory, a supercharacter theory with smaller superclasses.
\\

\begin{theorem}\label{GCI}
	Let $G$ be a group of order $n$ and $d\mathop{|}|G|$. If $\chi$ is a $d$-Galois character, then $\chi$ is a non-negative linear combination of the characters in $\mathsf{GIrr}_d(G)$, i.e., the  supercharacters of the supercharacter theory corresponding to
	$\mathsf{ Gal}{(n,d)},$ 
\end{theorem}

\begin{proof}
	Let  $\chi$ be a $d$-Galois character and let $\chi=m_1\psi_1+m_2 \psi_2+\ldots+m_t \psi_t$, where each $m_i$ is a positive integer and each $\psi_i$ is an irreducible character. Note that for every $\sigma\in {\rm Gal}{(n,d)}$, $$\chi=\sigma(\chi)=m_1 \sigma(\psi_1)+m_2 \sigma(\psi_2)+\cdots+ m_t\sigma(\psi_t),$$ and so for each $i$, $$\langle \chi, \psi_i \rangle=\langle \sigma(\chi),\sigma(\psi_i) \rangle= \langle \chi,\sigma(\psi_i) \rangle.$$ Therefore, for every $\sigma$ and $i$, $\sigma(\psi_i)$ is a constituent of $\chi$, and  the coefficients of $\psi_i$ and $\sigma(\psi_i)$ are the same positive integers. 
	As the supercharacters of the supercharacter theory corresponding to $\mathsf{Gal}{(n,d)}$ are of the form $\sum_{\phi\in \mathsf{O}_{d}(\psi)}\phi$ for an irreducible character $\psi$, we can  conclude that $\chi$ is a non-negative linear combination of the characters in $\mathsf{ GIrr}_d(G)$. 
\end{proof}

	In other words, the theorem above shows that for the set of all $d$-Galois characters,  {\it $d$-Galois irreducible characters} are playing a same role as the role of the irreducible characters in the set of all complex characters.   More precisely, the set of $d$-Galois irreducible character is
	$$\mathsf{ GIrr}_d(G)= \left\lbrace  \sum_{\phi\in \mathsf{O}_{d}(\psi)}\phi: \psi\in \mathsf{ Irr}(G) \right\rbrace.$$
	If $|G|\mathop{|}d$, by convention, we let $\mathsf{ GIrr}_d(G)$ be the set of all irreducible characters. 
	
\begin{remark}
  For any $d$-Galois irreducible character $\chi=\sum_{\phi\in \mathsf{O}_{d}(\psi)}\phi \in \mathsf{GIrr}_d(G)$, there is a constant $m_{d}(\psi)$ such that  $m_{d}(\psi)\chi$ is an irreducible character corresponding to an irreducible  module over $\mathbb{Q}[\zeta_d]$.   The constant $m_{d}(\psi)$ is called the Schur index of $\psi$ over $\mathbb{Q}[\zeta_d]$. For more details on Schur indices see \cite[Sections 9 and 10]{IS}.
\end{remark}

A supercharacter theory is called a {\it $d$-Galois supercharacter theory} if the images of the supercharacters lie on the field $\mathbb{Q}[\zeta_d]$. As a corollary of the last theorem we also have the following. 
\begin{corollary}\label{EK}
	Let $G$ be a group of order $n$, and let $d|n$. The finest $d$-Galois supercharacter theory is unique and is the one that comes from the action of $\mathsf{Gal}{(n,d)}$ on the irreducible characters and conjugacy classes, i.e., $$( \mathsf{GIrr}_d(G), \mathsf{GCl}_d ).$$
\end{corollary}
We call a function from $G$ to $\mathbb{C}$ a $d$-{\it Galois class function} if it is constant on the elements of  $\mathsf{GCl}_d$, and the set of all $d$-Galois class functions is denoted by $\mathsf{Gcf}_d(G)$. Note that $$\mathsf{Gcf}_d(G)=\mathbb{C}\text{-span}\{ \chi: \chi \in \mathsf{Glrr}_d(G) \}.$$

\begin{remark}{\rm
	$1$-Galois supercharacter theories have been studied by Keller \cite{Ke}, and he showed that the finest integral supercharacter theory is unique.  Corollary  \ref{EK} extends his result.}
\end{remark}

{\rm \begin{Example}
		\begin{enumerate}
			\item
			Let $q$ be a power of the prime number $p$. The  supercharacter theory for unipotent upper-triangular matrices over a finite field $\mathbb{F}_q$ introduced in \cite{DI08} is a $p$-Galois supercharacter theory.
			\item The nonnesting supercharacter theory introduced in \cite{An} is also a $p$-Galois supercharacter theory.
			\item Every normal supercharacter theory for $G$ is a  $1$-Galois supercharacter theory (see \cite{AL}).
		\end{enumerate}
\end{Example}}

\section{Positive self-dual Hopf algebras}\label{SPSH}

A graded connected Hopf algebra $\mathsf{H}$ over $\mathbb{C}$ with a distinguished $\mathbb{C}$-basis $\{ \beta_{\lambda} \}$ consisting of
homogeneous elements is a {\it  positive self-dual} Hopf algebra (or PSH) if it satisfies the two other axioms,
\begin{enumerate}
	\item ({\it self-duality}) The same structure constants $c_{\mu,\nu}^\lambda$ appear for the product $$m(\beta_\mu\otimes \beta_\nu)=\sum_{\lambda} c_{\mu,\nu}^\lambda \beta_\lambda$$ and the coproduct $$\Delta( \beta_\lambda )=\sum_{\mu,\nu} c_{\mu,\nu}^\lambda \beta_\mu \otimes \beta_\nu.$$
	\item ({\it positivity}) The structure constants $c_{\mu,\nu}^\lambda$ are all non-negative real numbers.
\end{enumerate}
The basis $\{ \beta_{\lambda} \}$ is called the {\it PSH-basis} of $\mathsf{H}$.
\\

Given a PSH $\mathsf{H}$ with PSH-basis $\Sigma=\{ \beta_\lambda \}$, we define a bilinear form $$\langle ., . \rangle: \mathsf{H}\otimes \mathsf{H} \rightarrow \mathbb{C}$$ on $\mathsf{H}$ that makes this basis orthonormal. Similarly, the elements $\{ \beta_\lambda \otimes \beta_\mu \}$ give an orthonormal basis for a form $\langle . , .\rangle_{\mathsf{H}\otimes \mathsf{H}}$  on $\mathsf{H} \otimes \mathsf{H}$. Consider that for  counit $\epsilon$ of   the PSH $\mathsf{H}=\bigoplus_{n\geq 0} \mathsf{H}_n$, we have $$I:=\ker \epsilon= \bigoplus_{n>0}\mathsf{H}_n$$ is a two sided ideal, and for any $x\in I$ we have $$\Delta(x)=1\otimes x+x\otimes 1+\Delta_+(x),$$ where $\Delta_+(x)\in I \otimes I$. Any element of $x\in I $ with $\Delta_+(x)=0$ is called a {\it primitive element}. 
Let $\mathfrak{p}$ be the $\mathbb{C}$-submodule of primitive elements inside $\mathsf{H}$. Note that the orthogonal complement of $m(I\otimes I)$ in $I$ is the set of elements $x$ such that for every $y\in I\otimes I$,
$$0=\langle x,m(y)\rangle_\mathsf{H}=\langle \Delta(x), y \rangle_{\mathsf{H}\otimes \mathsf{H}}=\langle \Delta_+(x), y \rangle_{\mathsf{H}\otimes \mathsf{H}}.$$
Since the above expression must be true for each $y\in I\otimes I$, we must have $\Delta_+(x)=0$. Therefore, the complement of $m( I \otimes I)$ in $I$ is $\mathfrak{p}$. Consequently, by basic linear algebra we have $I= \mathfrak{p}\oplus m(I \otimes I)$. Now by \cite[Appendix 1, Theorem AI.I]{ZE} we have $\mathsf{H}$ is isomorphic to the symmetric algebra of the space of its primitive elements $\mathsf{Sym}(\mathfrak{p})$. In other words: 

\begin{center}
	\fbox{\begin{minipage}{30em}
			\begin{center}
				The Hopf algebra $\mathsf{H}$ decomposes into the tensor product $$\mathsf{H}=\bigotimes_{p\in A} \mathsf{H}_p$$ where each Hopf algebra $\mathsf{H}_p$ has only one primitive element. Also, $\mathsf{H}_p$ is isomorphic to the algebra $K[x]$ of polynomials in one intermediate $x$, where ${\rm deg}(x)=p$ and $x$ is the primitive element.
				Therefore, $\mathsf{H}$ is unique up to isomorphism.
			\end{center}
	\end{minipage}}
\end{center}
Let $\mathcal{C}=\mathfrak{p}\cap \Sigma$,  the intersection of the $\mathbb{C}$-submodule of primitive elements $\mathfrak{p}$ and the PSH-basis $\Sigma$ of the PSH $\mathsf{H}$.  
Then $$\mathsf{H}\cong\bigotimes_{\rho\in \mathcal{C}}\mathsf{H}(\rho),$$ where $\mathsf{H}(\rho)$ is the $\mathbb{C}$-module spanned by $$ \Sigma(\rho):=\{ \beta\in \Sigma: \text{~ there exists $n\geq 0$ with $\langle \beta,\rho^n\rangle\neq 0$} \}.$$

\begin{remark}
	{\rm 
	{\rm (1)}	In \cite{ZE}, Zelevinsky developed a theory for self-dual Hopf algebra over $\mathbb{Z}$, where the constant coefficients of product and coproduct for a distinguished basis $\Sigma$ are non-negative integers. We now briefly mention Zelevinsky's theory. Given such a self-dual Hopf algebra $\mathsf{H}$ with the distinguished basis $\Sigma$, let $\mathcal{C}$ be the intersection of the basis $\Sigma$ and the $\mathbb{Z}$-submodule of primitive elements. For each $\rho\in \mathcal{C}$, let $\mathsf{H}(\rho)$ be the $\mathbb{Z}$-span of $$ \Sigma(\rho):=\{ \beta\in \Sigma: \text{~ there exists $n\geq 0$ with $\langle \beta,\rho^n\rangle\neq 0$} \}.$$ Each $\mathsf{H}(\rho)$ is isomorphic to the Hopf algebra of symmetric functions $\mathsf{Sym}$. The Zelevinsky's theorem \cite[Theorem 2.2]{ZE} shows that $\mathsf{H}$ has a canonical tensor product  decomposition
	$$\mathsf{H}=\bigotimes_{\rho\in \mathcal{C}}\mathsf{H}(\rho).$$ 
	
	\noindent {\rm (2)}	Despite Zelevisnsky's structure theory, in the theory of positive self-dual Hopf algebras over $\mathbb{C}$ that we mentioned earlier, $\mathsf{H}(\rho)$ is not necessarily isomorphic to the ring of symmetric functions $\mathsf{Sym}$. Moreover, $\mathsf{H}(\rho)$ is not unique up to isomorphism and it depends on the PSH-basis for $\mathsf{H}$.   For example if the power sum basis 
	$$\Sigma=\{ p_\lambda :\lambda  \text{~is a partition} \}$$ is the PSH-basis for $\mathsf{Sym}.$ Then
	$$\mathcal{C}=\Sigma\cap \mathfrak{p}=\{p_n: n\geq 1  \},$$ and
	$$\mathsf{Sym}(p_n)=\mathbb{C}\text{-span}\{ p_{(n^k)}:k\geq 1 \},$$ where $(n^k)$ is a partition of $kn$, with $k$ parts all equal to $n$. Consider that $\mathsf{Sym}(p_n)$ is not the ring of symmetric functions. However, if we choose the set of Schur functions $\{ S_\lambda: \lambda \text{~is a partition} \}$ as the PSH-basis for $\mathsf{Sym}$, then $\mathcal{C}=\{ p_1 \}$, and $\mathsf{Sym}(p_1)=\mathsf{Sym}$.}
\end{remark}

\section{Some Hopf algebras from Galois character theory}\label{Hopf algebra}
The character algebras of some families of groups  are endowed with Hopf structures; general linear groups and unipotent upper-tringular matrices have this property (see \cite{ABT} and \cite{GR}). In this section, we study the properties that the Hopf algebra of $d$-Galois class functions  inherits from the Hopf algebra of class functions. We will show that if the Hopf algebra of class functions is a PSH, then so is the Hopf algebra of  $d$-Galois class functions.
\\

Consider a tower of groups
$$G_*=( G_0< G_1<G_2<G_3< \cdots ).$$
Let $$\mathsf{H}(G_*)=\bigoplus_{n\geq 0} {\mathsf{H}(G_n)},$$ where $$\mathsf{H}({G_n})=\mathbb{C}\text{-span}\{ \psi:\psi\in\mathsf{ Irr}(G_n) \}.$$

 Let $H$ be a subgroup of a given group $G$, and $N$ a normal subgroup of $G$. Let $\chi$ be a character of $G$, $\psi$ a character for $H$, and $\phi$ a character of $G/N$. Consider the following functions,
%

%
%
%

\begin{equation}\label{product}
\left( {\rm Ind}_{H}^{G} \psi\right) (g)=\frac{1}{|H|}\sum_{ \substack{k\in G: \\ kgk^{-1}\in H}}\psi(kgk^{-1}) \hspace{1cm}  \left( {\rm Infl}_{G/N}^{G}\phi\right) (g)=\phi(gN)
\end{equation}
\begin{equation}\label{coproduct}
\left( {\rm Res}_{H}^{G}\chi \right) (h)=\chi(h)  \hspace{2cm} \left( {\rm Defl}_{G/N}^G \chi \right)(g)=\frac{1}{|N|}\sum_{k\in gN}\chi(k)
\end{equation}

%

Note that each of the above functions maps a $d$-Galois character to a $d$-Galois character, and moreover, each of the functions yields a functor from the categoty of modules over one group to the category of modules over the other one.
\begin{theorem}\label{PSH}
	Let $G_*$ be a tower of groups and let $d$ be a positive integer such that for some $k$, $G_k\mathop{|}d$ and $d\mathop{|}G_{k+1}$. If $\mathsf{H}(G_*)$ is a graded connected Hopf algebra with a linear extension of compositions of the functions in \ref{product} as product and a  linear extension of compositions of the functions in \ref{coproduct} as coproduct, then 
	\begin{enumerate}
		\item  $$\mathsf{GH}_d(G_*)=\bigoplus_{n\geq 0} {\mathsf{GH}_d(G_n)},$$ where $$\mathsf{GH}_d({G_n})=\mathbb{C}\text{\rm-Span}\{ \psi:\psi\in\mathsf{ GIrr}_d(G_n) \},$$  is a Hopf algebra with the same product and coproduct as $\mathsf{H}(G_*)$. 
		\item If $\mathsf{H}(G_*)$ is a PSH with PSH-basis $\mathsf{Irr}(G)$, then $\mathsf{GH}_d(G_*)$ is a PSH with PSH-basis $$\Sigma_d=\left\lbrace  1/\sqrt{\mathsf{O}_d(\psi)} \sum_{\phi\in \mathsf{O}_d(\psi)}\phi: \psi\in \bigsqcup_{n\geq 0} \mathsf{Irr}(G_n) \right\rbrace .$$
	\end{enumerate}
\end{theorem}

\begin{proof} Consider that 
	$$\left\lbrace  1/\sqrt{\mathsf{O}_d(\psi)} \sum_{\phi\in \mathsf{O}_d(\psi)}\phi: \psi\in \bigsqcup_{n\geq 0} \mathsf{Irr}(G_n) \right\rbrace=
	\left\lbrace 1/\sqrt{\langle \chi,\chi \rangle } \chi: \chi \in \bigsqcup_{n\geq 0} \mathsf{GIrr}_d(G_n)     \right\rbrace  .$$
 Let $\psi\in \mathsf{ GIrr}_d (G_i)$ and $\varphi\in \mathsf{GIrr}_d (G_j)$.	 
  We have
 $$m\left( \left( 1/\sqrt{\langle \psi,\psi \rangle}\right)  \phi \otimes \left( 1/\sqrt{\langle \varphi,\varphi\rangle }\right)  \varphi\right) =$$
 $$ 
 \left( 1/\sqrt{\mathsf{O_{d}}(\psi)  }\right) \left(  1/\sqrt{\mathsf{O_{d}}(\varphi)  }  \right) m\left(   \psi \otimes \varphi \right)
 $$
is a non-negative  $\mathbb{R}$-linear combination of $d$-Galois irreducible characters. Similarly, 
$$\Delta \left(  1/\sqrt{\mathsf{O}_{d}(\chi)  } \sum_{\phi\in \mathsf{O}_{d}(\chi) } \phi \right)$$ is a non-negative $\mathbb{R}$-linear combination of $d$-Galois irreducible characters. That exhausts the  positivity.
 It remains to show that $\mathsf{GH}_d(G_*)$ is self dual,	i.e., if for any $\psi\in \mathsf{GIrr}_d (G_i)$, $\phi\in \mathsf{GIrr}_d (G_j)$, and $\chi\in \mathsf{ GIrr}_d (G_{i+j})$,

	$$m( \left( 1/\sqrt{\langle \psi,\psi \rangle }\right)  \psi \otimes \left( 1/\sqrt{\langle \phi,\phi \rangle }\right)  \phi )=\sum c_{\psi,\phi}^\chi \left( 1/\sqrt{\langle \chi,\chi \rangle }\right) \chi,$$ and $$\Delta(\left( 1/\sqrt{\langle \chi,\chi \rangle }\right)\chi)=\sum d_{\psi,\phi}^\chi \left(  1/\sqrt{\langle \psi,\psi \rangle }\right) \psi\otimes \left(  1/\sqrt{\langle \phi,\phi \rangle }\right) \phi,$$ then $$c_{\psi,\phi}^\chi=d_{\psi,\phi}^\chi.$$
	
	We have that $$c_{\psi,\phi}^\chi \left< \left( 1/\sqrt{\langle \chi,\chi \rangle }\right)\chi,\left( 1/\sqrt{\langle \chi,\chi \rangle }\right)\chi \right> =$$
	$$\left< m(\left( 1/\sqrt{\langle \psi,\psi \rangle }\right)\psi \otimes \left( 1/\sqrt{\langle \phi,\phi \rangle }\right)\phi), \left( 1/\sqrt{\langle \chi,\chi \rangle }\right)\chi \right>=$$
	$$\left< \left( 1/\sqrt{\langle \psi,\psi \rangle }\right)\psi \otimes \left( 1/\sqrt{\langle \phi,\phi \rangle }\right)\phi, \Delta(\left( 1/\sqrt{\langle \chi,\chi \rangle }\right)\chi) \right>=$$
	$$d_{\psi,\phi}^\chi  \left< \left( 1/\sqrt{\langle \psi,\psi \rangle }\right)\psi \otimes \left( 1/\sqrt{\langle \phi,\phi \rangle }\right)\phi,\left( 1/\sqrt{\langle \psi,\psi \rangle }\right)\psi \otimes \left( 1/\sqrt{\langle \phi,\phi \rangle }\right)\phi \right>.$$ Note that 
	$$\left< \left( 1/\sqrt{\langle \chi,\chi \rangle }\right)\chi,\left( 1/\sqrt{\langle \chi,\chi \rangle }\right)\chi\right>=$$
	$$\left< \left( 1/\sqrt{\langle \psi,\psi \rangle }\right)\psi \otimes \left( 1/\sqrt{\langle \phi,\phi \rangle }\right)\phi,\left( 1/\sqrt{\langle \psi,\psi \rangle }\right)\psi \otimes \left( 1/\sqrt{\langle \phi,\phi \rangle }\right)\phi \right>=1.$$  We conclude that $c_{\psi,\phi}^\chi=d_{\psi,\phi}^\chi.$  Therefore,
	%
%
	$\mathsf{GH}_d(G_*)$ is a PSH with the PSH-basis $$	\left\lbrace 1/\sqrt{\langle \chi,\chi \rangle } \chi: \chi \in \bigsqcup_{n\geq 0} \mathsf{GIrr}_d(G_n)     \right\rbrace  =\left\lbrace  1/\sqrt{\mathsf{O}_d(\psi)} \sum_{\phi\in \mathsf{O}_d(\psi)}\phi: \psi\in \bigsqcup_{n\geq 0} \mathsf{Irr}(G_n) \right\rbrace .$$
\end{proof}

\begin{Example}

Consider the following towers of groups
$$G_*=( G_0< G_1<G_2<G_3< \cdots )$$
where either
\begin{itemize}
	\item $G_n=\mathfrak{S}_n[\Gamma]$, the wreath product of the symmetric group with some arbitrary finite group $\Gamma$,
	\item $G_n=GL_n(\mathbb{F}_q)$, the finite general linear group over $\mathbb{F}_q$, or
	\item $G_n=UT_n(\mathbb{F}_q)$ the group of unipotent upper-triangular matrices over $\mathbb{F}_q$.
\end{itemize}

The set of class functions of each of these towers carries out a Hopf structure. For the Hopf algebra structures of the first two see \cite{GR,ZE}, and for the Hopf algebra of the last one see \cite{AABB12,ABT, BT}. Therefore, the set of $d$-Galois class functions of the groups in towers together also have a Hopf structure. Moreover, the first two yield a PSH, and so do their $d$-Galois class functions. \end{Example}

In the next Section, we describe product and coproduct of the Hopf algebra of characters of general linear groups as the focus of this paper is more on the  study of the Galois characters of general linear groups.

\section{The Hopf structure of Galois class functions of $GL_n(\mathbb{F}_q)$}

Fix a prime power $q$. We frequently use the notation $GL_n$ instead of $GL_n(\mathbb{F}_q)$. From now on, fix $d$ in a way that $|GL_k|\mathop{|}d$ and $d\mathop{|}|GL_{k+1}|$ for some non-negative integer $k$.  Let $$\mathsf{cf}(GL_*)=\bigoplus_{n\geq 0}\mathsf{cf}(GL_n),$$ where $$\mathsf{cf}(GL_n)=\mathbb{C}\text{-span}\{ \chi: \chi \in \mathsf{Irr}(GL_n)\}.$$
Embed $GL_i\times GL_j$ into $GL_{i+j}$ as block diagonal matrices whose two diagonal blocks have size $i$ and $j$, respectively. For instance, if $g_i\in GL_i$ and $g_j\in GL_j$, the we map $(g_i,g_j)$ to $$\left[ \begin{array}{cc}g_i&0\\0&g_j\end{array}\right].$$
 Let $P_{i,j}$ denote the prabolic subgroup consisting of the block upper-triangular matrices of the form $$\left[ \begin{array}{cc}g_i&\ell\\0&g_j\end{array}\right]$$ where  $g_i\in GL_i$, $g_j\in GL_j$, and $\ell$ is an $i\times j$ matrix with entries in $\mathbb{F}_q$. Note that $GL_i\times GL_j$ is isomorphic to the quotient $P_{i,j}/U_{i,j}$ where $U_{i,j}$ is the set of matrices of the form $$\left[ \begin{array}{cc}I_i&\ell\\0&I_j\end{array}\right].$$
   For $\chi$, $\psi$, and $\phi$ characters of $GL_{l}$, $GL_i$, $GL_j$, respectively, define
$$m(\psi \otimes \phi)={\rm Ind}_{P_{i,j}}^{GL_{i+j}} {\rm Inf}_{GL_i\times GL_j}^{P_{i,j}} (\psi\otimes \phi)$$
$$\Delta(\chi)=\sum_{i+j=l}{\rm Defl}_{GL_i\times GL_j}^{P_{i,j}} {\rm Res}_{P_{i,j}}^{GL_{i,j}} \chi .$$
Extend both of the functions $m$ and $\Delta$ $\mathbb{C}$-linearly. It is well-known that $\mathsf{cf}(GL_*)$ is a positive self-dual Hopf algebra with the above product and coproduct (see \cite{ZE}).

 Let $$\mathsf{Gcf}_d(GL_*)=\bigoplus_{n\geq 0}\mathsf{Gcf}_d(GL_n),$$ where $$\mathsf{Gcf}_d(GL_n)=\mathbb{C}\text{-span}\{ \chi: \chi \text{~is a ~}d\text{-Galois irreducible character of~}GL_n \}.$$
Then by Theorem \ref{PSH}, and the fact that $\mathsf{cf(GL_*)}$ is a PSH, we have the following corollary.
\begin{corollary}
	$\mathsf{Gcf}_d(GL_*)$ is a PSH.
\end{corollary}

We conclude this section with finding all possible numbers $d$ to construct $\mathsf{Gcf}_d(GL_*)$. Let $|\mathbb{F}_q|=p^r$. Consider that $|GL_n|=q^{n-1} q^{n-2} \ldots q (q^n-1)(q^{n-1}-1)\ldots(q-1)$. With a simple calculation it can be seen that 
$$\frac{|GL_n|}{|GL_{n-1}|}=p^{r(n-1)}(p^{rn}-1).$$
Note that the $n$th cyclotomic polynomial is 
$$\varPhi_n(p)=\prod_{ \substack{ 1\leq k \leq n\\ gcd(k,n)=1} }(p-e^{2i\pi\frac{k}{n}}).$$ All coefficients of $\varPhi_n(p)$ are in $\mathbb{Z}$, and moreover, if $D(rn)$ is the set of divisors of $rn$, then  $$p^{rn}-1=\prod_{s\in D}\varPhi_s(p).$$ 
Now, the set of all choices for $d$ to construct $\mathsf{Gcf}_d(GL_*)$ is
$$\left\lbrace  p^{r(n-1)+i} \prod_{s\in A}\varPhi_s(p): A\subseteq D, n\in \mathbb{N}, 1\leq  i\leq r  \right\rbrace.$$
It is worth mentioning  that we always can construct $\mathsf{Gcf}_1(GL_*)$ independent of the field $\mathbb{F}_q$.

\section{Classification of the Galois classes of $GL_n(\mathbb{F}_q)$}
\subsection{Classification of the Galois classes}
  In this section, we will study and parametrize the  Galois classes of the general linear groups. From now on most of our notations are the same as \cite[Chapter IV]{MC}. Fix an algebraic closure of the field $k=\mathbb{F}_q$.
Let $\overline{k}$ be the algebraic closure of $k$. Let $F$ be the Frobenius automorphism of $\overline{k}$ over $k$,
$$
\begin{array}{cccc}
F:& \overline{k}&\rightarrow & \overline{k}\\
& x & \mapsto & x^q.
\end{array}
$$
 For each $n\geq 1$, let $k_n$ be the set of fixed points of $F^n$ in $\overline{k}$, i.e, $$k_n=\{ x\in \overline{k}: F^n(x)=x \}=\{ x\in \overline{k}: x^{q^n}=x \},$$ and note that $k_n$ is isomorphic to $\mathbb{F}_{q^n}$.

Let $\langle F \rangle$ be the cyclic group of the powers of $F$. The group  $\langle F \rangle$ acts on the group $\overline{k}^\times$ by $F^n.x=F^n(x)=x^{q^n}$.

Let $$\Phi=\{ \langle F\rangle\text{-orbits in~} \overline{k}^\times \}.$$
Any monic irreducible polynomial $f$ over $k$ can be uniquely written as
$$f=\prod_{i=0}^{{{s}}-1}(t-x^{q^i}),$$ where $\{ x,x^q,\ldots,x^{q{{s}-1}} \}$ is an element of $\Phi$. So we identify $\Phi$ with the set of monic irreducible polynomials over $k$.

Let $\mathscr{P}$ be the set of all partitions. Any conjugacy class of $GL_n$ is uniquely indexed by a function $$\boldsymbol{\mu}:\Phi \rightarrow  \mathscr{P}$$  with $$\| \boldsymbol{\mu} \|=\sum_{f\in \Phi}  |\boldsymbol{\mu}(f)| {\rm deg}(f)=n.$$
Let $f=t^s-\sum_{i=1}^{s-1}a_it^{i-1}\in \Phi$, and let 
$$J(f)=\left( \begin{array}{ccccc}
0&1&0&\cdots&0\\
0&0&1&\cdots&0\\
\cdots& &&&\cdots\\
0&0&0& \cdots&1\\
a_1&a_2&a_3 & \cdots&a_s
\end{array}\right)$$ the companion matrix for the polynomial $f$.
For any integer $m\geq 1$, define
$$J_m(f)=\left( \begin{array}{ccccc}
J(f)&I_s&0&\cdots&0\\
0&J(f)&I_s&\cdots&0\\
\cdots& &&&\cdots\\
0&0&0& \cdots&I_s\\
0&0&0 & \cdots&J(f)
\end{array}\right).$$ 
  Let $\mu=(\mu_1,\mu_2,\ldots,\mu_{l(\mu)})$ be a partition. Define $$J_\mu(f)=\bigoplus_{i=1}^{l(\mu)}J_{\mu_i}(f).$$ The Jordan canonical form $g_{\boldsymbol{\mu}}$ of a representative of the conjugacy class $c_{\boldsymbol{\mu}}$ indexed by $\boldsymbol{\mu}: \Phi \rightarrow \mathscr{P}$ is 
$$g_{\boldsymbol{\mu}}=\bigoplus_{f\in \Phi}J_{{\boldsymbol{\mu}(f)}}(f)$$

As the Galois classes are the orbits of the action of Galois groups on conjugacy classes $c_{\boldsymbol{\mu}}$, we now identify these orbits.
Before that we need the following lemma.

 \begin{lemma}
	Let $J_n(\lambda)$ denote the $n\times n$ Jordan block matrix with eigenvalue $\lambda$. For each positive integer $l$, $J_n(\lambda)^l$ is the matrix whose $(i,j)$ entry $a_{i,j}$ is given by \[a_{i,j}=\begin{cases}\lambda^{l-(j-i)}\binom{l}{j-i}&\text{if $j\ge i$}\\
	0&\text{otherwise.}\end{cases}\]
\end{lemma}
\begin{proof}
	Let $E_i$ denote the matrix with ones along the $i^{\rm th}$ superdiagonal and zeros elsewhere. Since $E_1^k=E_k$ for all $k$,
	\[J_n(\lambda)^l=(\lambda I_n+E_1)^l=\sum_{k=0}^l\lambda^{l-k}E_1^k=\sum_{k=0}^l\binom{l}{k}\lambda^{l-k}E_k.\]
	Thus, for each $(i,j)$ with $j-i=k$, the $(i,j)$ entry of $J_n(\lambda)^l$ is the $(i,j)$ entry of $\binom{l}{k}\lambda^{l-k}E_k$, which is $\binom{l}{k}\lambda^{l-k}$.
\end{proof}

\begin{proposition}	\label{powers} Let $\lambda$ be in $\overline{k}$. The Jordan form of $J_n(\lambda)^l$ is $J_n(\lambda^l)$ if and only if $l$ is prime to $p=char(k)$.
\end{proposition}
\begin{proof}
	Since the only eigenvalue of $J_n(\lambda)^l$ is $\lambda^l$, the Jordan form of  $J_n(\lambda)^l$ is $J_n(\lambda^l)$ if and only if the minimal polynomial of $J_n(\lambda)^l$ is $(x-\lambda^l)^n$. Now, since $J_n(\lambda)^l=\sum_{k=0}^l\binom{l}{k}\lambda^{l-k}E_k$ and the minimal polynomial of $E_1$ is $x^n$, the minimal polynomial of $J_n(\lambda)^l$ will be $(x-\lambda^l)^n$ if and only if $0\neq\binom{l}{1}=l$, which happens if and only if $p$ does not divide $l$.
\end{proof}

Let $\sigma\in \mathsf{Gal}{(|GL_n|,d)}$ and $r$ be the positive  integer such that $\sigma(\zeta_{|GL_n|})=\zeta_{|GL_n|}^r$.
Define an action of $\mathsf{Gal}{(|GL_n|,d)}$ on $\overline{k}$ by $\sigma.x=x^r$, and 
for any monic irreducible polynomial $f=\prod_{i=1}^{s-1}(t-x^{q^i})$ define $$\sigma.f=\prod_{i=1}^{s-1}(t-(\sigma.x)^{q^i})=\prod_{i=1}^{s-1}(t-(x^r)^{q^i}).$$ 

 For any $\sigma\in \mathsf{Gal}{(|GL_n|,d)}$ and a partition valued function $\boldsymbol{\mu}$ define
$$
\begin{array}{cccc}
\sigma.\boldsymbol{\mu}:& \Phi& \rightarrow& \mathscr{P}\\
& f & \mapsto& \boldsymbol{\mu}(\sigma^{-1}.f).
\end{array}$$
Let $$\mathsf{O}_d(\boldsymbol{\mu})=\{ \sigma. \boldsymbol{\mu}: \sigma\in \mathsf{Gal}{(|GL_n|,d)}\}.$$

In the following theorem we determine when two conjugacy classes of $GL_n$ are in the same Galois class.
\begin{theorem}
	Let $\boldsymbol{\mu}$, $\boldsymbol{\nu}:\Phi\to\mathscr{P}$ with $\lVert{\boldsymbol{\mu}}\rVert=\lVert{\boldsymbol{\nu}}\rVert=n$. Then the classes $c_{\boldsymbol{\mu}}$ and $c_{\boldsymbol{\nu}}$ lie in the same $d$-Galois class if and only if $$\mathsf{O}_d(\boldsymbol{\mu})=\mathsf{O}_d(\boldsymbol{\nu}).$$ 
\end{theorem}

\begin{proof}
	Two conjugacy classes $c_{\boldsymbol{\mu}}$ and $c_{\boldsymbol{\nu}}$ are in the same $d$-Galois class if and only if for some $\sigma\in \mathsf{Gal}{(|GL_n|,d)}$, $\sigma.c_{\boldsymbol{\mu}}=c_{\boldsymbol{\nu}}$. The conjugacy class $c_{\boldsymbol{\mu}}$ contains the element $g_{\boldsymbol{\mu}}$. Note that 
	$$\sigma.g_{\boldsymbol{\mu}}= g_{\boldsymbol{\mu}}^r=\left( \bigoplus_{f\in \Phi}J_{{\boldsymbol{\mu}(f)}}(f)\right)^r,$$ where $r$ is the positive integer such that $\sigma(\zeta_{|GL_n|})=\zeta_{|GL_n|}^r$.
	 We have that $$\left( \bigoplus_{f\in \Phi}J_{{\boldsymbol{\mu}(f)}}(f)\right)^r=\left( \bigoplus_{f\in \Phi}{\left( J_{{\boldsymbol{\mu}(f)}}(f)\right) }^r\right).$$
	Also, 
	$${\left( J_{{\boldsymbol{\mu}(f)}}(f)\right) }^r=\bigoplus_{i=1}^{l(\boldsymbol{\mu}(f))}\left( J_{\boldsymbol{\mu}_i(f)}(f)\right)^r.$$ Moreover, by Proposition \ref{powers}, 
	$$\left( J_{\boldsymbol{\mu}_i(f)}(f)\right)^r=\left( \begin{array}{ccccc}
	J(f)^r&I&0&\cdots&0\\
	0&J(f)^r&I&\cdots&0\\
	\vdots& &&&\vdots\\
	0&0&0& \ddots&I\\
	0&0&0 & \cdots&J(f)^r
	\end{array}\right).$$
Let $f=\prod_{i=1}^{s-1}(t-x^{q^i})$, where $x^{q^{s}}=1$. The Jordan form of $J(f)^r$ is 
$$\left( \begin{array}{ccccc}
x^r&1&0&\cdots&0\\
0&(x^q)^r&1&\cdots&0\\
\vdots& &&\ddots &\vdots\\
0&0&0&\cdots&1\\
0&0&0&\cdots&(x^{q^{d-1}})^r
\end{array}\right).$$ Therefore, $J(f)^r=J(\sigma.f)$.
We have $$\sigma.g_{\boldsymbol{\mu}}= \bigoplus_{f\in \Phi}J_{{\boldsymbol{\mu}(f)}}(\sigma.f)$$
Note that  $$g_{\sigma.\boldsymbol{\mu}}= \bigoplus_{f\in \Phi}J_{{\sigma.\boldsymbol{\mu}(f)}}(f)=\bigoplus_{f\in \Phi}J_{{\boldsymbol{\mu}(\sigma^{-1}.f)}}(f)=\bigoplus_{\sigma.f\in \Phi}J_{{\boldsymbol{\mu}(f)}}(\sigma.f)= \bigoplus_{f\in \Phi}J_{{\boldsymbol{\mu}(f)}}(\sigma.f)=\sigma.g_{\boldsymbol{\mu}}.$$
Therefore, $g_{\sigma.\boldsymbol{\mu}}$ is in $c_{\boldsymbol{\nu}}$, we have ${\sigma.\boldsymbol{\mu}}={\boldsymbol{\nu}}$, and so $$\mathsf{O}_d(\boldsymbol{\mu})=\mathsf{O}_d(\boldsymbol{\nu}).$$ 
\end{proof}

\subsection{Characteristic map} Let $\pi_{\boldsymbol{\mu}}$ be the the characteristic function of the conjugacy class $c_{\boldsymbol{\mu}}$, i.e.,
$$\pi_{\boldsymbol{\mu}}(g)=\begin{cases}
1 & g\in c_{\boldsymbol{\mu}},\\
0& \text{otherwise.}
\end{cases}$$

Let $X_{i,f}$ ($i\geq 1,f\in \Psi$) be independent variables over $\mathbb{C}$. Define $e_n(f)$ to be the elementary symmetric function on the variables $X_{i,f} (i\geq 1)$.
Let $$\mathsf{B}=\mathbb{C}[ e_n(f) : n\geq 1, f\in \Phi].$$
Let $\lambda$ be a partition and $t$ an indeterminate, and let  $P_\lambda(X_f;t)$ be the Hall-Littlewood polynomial. For each $\boldsymbol{\mu}:\Phi\to\mathscr{P}$ such that $\| \boldsymbol{\mu} \|<\infty$ let
$$\tilde{P}_{\boldsymbol{\mu}}=\prod_{f\in \Psi}\tilde{P}_{\boldsymbol{\mu}(f)}(f),$$ where
$\tilde{P}_\lambda(f)=q_f^{-n(\lambda)} P_\lambda(X_f;q_f^{-1})$ and $n(\lambda)=\sum_{i\geq 1}(i-1)\lambda_i$. It is well-known that the characteristic map 
$$\begin{array}{cccc}
ch:&A&\rightarrow & B\\
&\pi_{\boldsymbol{\mu} }&\mapsto&\tilde{P}_{\boldsymbol{\mu}}
\end{array}$$ is an isomorphism (see \cite[(4.1)]{MC}). 
 Let
$$\tilde{P}_{{\mathsf{O}_d(\boldsymbol{\mu})}}=\sum_{\boldsymbol{\nu}\in {\mathsf{O}_d(\boldsymbol{\mu})}}\tilde{P}_{\boldsymbol{\nu}}$$ and

$$ \mathsf{B}_d=\mathbb{C}{\text{-span}}\{ \tilde{P}_{{\mathsf{O}_d(\boldsymbol{\mu})}} : \boldsymbol{\mu} \text{~is a function from $\Phi$ to $\mathscr{P}$} \}.$$
%
%
Let  $\pi_{\mathsf{O}_d(\boldsymbol{\mu})}$ be the characteristic function of the $d$-Galois class containing $c_{\boldsymbol{\mu} }$.

\begin{theorem}
	The map $$
	\begin{array}{cccc}
	ch_d: & \mathsf{Gcf}_d(GL_*) &\rightarrow &\mathsf{B}_d\\
	& \pi_{\mathsf{O}_d(\boldsymbol{\mu})}&\mapsto& \tilde{P}_{{\mathsf{O}_d(\boldsymbol{\mu})}}
	\end{array}$$ is a Hopf isomorphism.
\end{theorem}

\begin{proof} 
	Consider that $ch_d$ is the restriction of $ch$ to $\mathsf{Gcf}_d(GL_*)$, and  is therefore surjective. Since 
	 $$ch(\pi_{\mathsf{O}_d(\boldsymbol{\mu})})=ch(\sum_{\boldsymbol{\nu}\in \mathsf{O}_d(\boldsymbol{\mu})} \pi_{\boldsymbol{\nu}})=\sum_{\boldsymbol{\nu}\in\mathsf{O}_d(\boldsymbol{\mu})}ch( \pi_{\boldsymbol{\nu}})=\sum_{\boldsymbol{\nu}\in \mathsf{O}_d(\boldsymbol{\mu})}\tilde{P}_{\boldsymbol{\nu}}=\tilde{P}_{\mathsf{O}_d(\boldsymbol{\mu})}.$$  It follows that $ch_d$ is surjective..
\end{proof}

\section{Classification of the Galois characters of $GL_n(\mathbb{F}_q)$}
\subsection{Classification of the Galois characters}

We will classify the Galois characters of $GL_n$. If $ms=n$, define the following surjective homomorphism,
$$\begin{array}{cccc}
N_{n,m}:& k_n^\times& \rightarrow& k_m^\times\\
& x &\mapsto & \prod_{i=0}^{s-1}x^{q^{mi}}
\end{array}.$$ Let
$$\begin{array}{cccc}
\widehat{N}_{n,m}:& {\rm Hom}(k_m^\times,\mathbb{C}^\times)& \rightarrow& {\rm Hom}(k_n^\times,\mathbb{C}^\times)\\
& \xi&\mapsto & \xi\circ N_{n,m}
\end{array}.$$
Note that $\widehat{N}_{n,m}$ is an injection. By this identification, we define $$L:=\bigcup_{n\geq 1}{\rm Hom}(k_n^\times,\mathbb{C}^\times).$$ 

\begin{remark}
	The groups ${\rm Hom}(k_n^\times,\mathbb{C}^\times)$  and the homomorphisms $\widehat{N}_{n,m}$ form a direct system, and $L$ is the direct limit of this system. Moreover, $L$ is isomorphic to the group of all roots of unity in $\mathbb{C}$ of order prime to $p$. 
\end{remark}
%
%
%

Let $F$ be the Frobenius map on $L$.
The group $ \langle F\rangle$ acts on $L$ as follows,
$$\begin{array}{ccc}
\langle F\rangle \times L & \rightarrow &L\\
(F^s , \xi) & \mapsto & \xi^{q^s},
\end{array}$$
where
$\xi^{q^s}(x)=\xi(x^{q^s})$.
 
Let $\Theta$ be the set of all $\langle F\rangle$-orbits in $L$.
Every irreducible character of $GL_n$ is uniquely indexed by a function $$\boldsymbol{\lambda}:\Theta \rightarrow  \mathscr{P},$$  with $$\| \boldsymbol{\lambda} \|=\sum_{\varphi \in \Theta}  |\boldsymbol{\lambda}(\varphi)| |\varphi|=n.$$

For every $\sigma\in \mathsf{Gal}{(|GL_n|,d)}$ and $\varphi \in \Theta$, define $$\sigma.\varphi=\{\sigma.\xi: \xi \in \varphi\},$$ where $\sigma.\xi(x)=\xi(\sigma.x)=\xi(x^r)$ such that $\sigma(\zeta_{GL_n})=\zeta_{GL_n}^r$. Let $$\mathsf{O}_d(\varphi)=\{\sigma.\varphi: \sigma\in \mathsf{Gal}(|GL_n|,d) \}.$$

For any function ${\boldsymbol{\lambda}}: \Theta \rightarrow \mathscr{P}$ with $\|\boldsymbol{\lambda} \|=n$ and $\sigma\in \mathsf{Gal}{(|GL_n|,d)}$, define $\sigma.\boldsymbol{\lambda}$ as follows,
$$\begin{array}{cccc}
\sigma.\boldsymbol{\lambda}:& \Theta& \rightarrow& \mathscr{P}\\
& \varphi & \mapsto & \sigma.\varphi.
\end{array}$$
Extend the action of $\mathsf{Gal}{(|GL_n|,d)}$ on $\{ \chi^{\boldsymbol{\lambda} }:  \|  \boldsymbol{\lambda}\|=n  \}$ linearly to $\mathsf{cf}(GL_n)$. 
We now see that the actions of Galois group on the product and coproduct of $\mathsf{cf}(GL_*)$ are compatible.
\begin{proposition}\label{pro-copro}
	For any $\chi\in \mathsf{Irr}(GL_{k})$, $\psi\in \mathsf{Irr}(GL_i)$, $\phi\in\mathsf{Irr}(GL_j)$, $\sigma\in \mathsf{Gal}{(GL_{i+j},d)}$, and $\tau\in \mathsf{Gal}{(|GL_{k}|,d)}.$ We have 
	$$\sigma.(m(\chi\otimes \phi))=m(\sigma.\chi\otimes \sigma.\psi) \qquad\text{and}\qquad\Delta(\tau.\chi)=\tau.\Delta(\chi).$$	
\end{proposition}

\begin{proof}
	Consider that $$\sigma.m(\psi \otimes \phi)=\sigma.\left( {\rm Ind}_{P_{i,j}}^{GL_{i+j}} {\rm Inf}_{GL_i\times GL_j}^{P_{i,j}} (\psi\otimes \phi)\right) ={\rm Ind}_{P_{i,j}}^{GL_{i+j}} {\rm Inf}_{GL_i\times GL_j}^{P_{i,j}} (\sigma.(\psi\otimes \phi))=$$$${\rm Ind}_{P_{i,j}}^{GL_{i+j}} {\rm Inf}_{GL_i\times GL_j}^{P_{i,j}} (\sigma.\psi\otimes \sigma.\phi)=m(\sigma.\psi\otimes \sigma.\phi).$$ Also,

	$$\Delta(\tau.\chi)=\sum_{i+j=k}{\rm Defl}_{GL_i\times GL_j}^{P_{i,j}} {\rm Res}_{P_{i,j}}^{GL_{i,j}} \tau.\chi=
\tau.\sum_{i+j=k}{\rm Defl}_{GL_i\times GL_j}^{P_{i,j}} {\rm Res}_{P_{i,j}}^{GL_{i,j}} \chi=\tau.\Delta(\chi).$$
\end{proof}

We have the following lemma.

\begin{lemma}
	For any $\sigma\in \mathsf{Gal}{(|GL_n|,d)}$ and $\boldsymbol{\mu}\in \Phi$ and $\|  \boldsymbol{\mu}\|=n$, we have $$\sigma.\pi_{\boldsymbol{\mu} }=\pi_{\sigma^{-1}. \boldsymbol{\mu} }.$$
\end{lemma}

\begin{proof}
	Let $$\pi_{\boldsymbol{\mu}}=\sum_{    \|\boldsymbol{\lambda}\|=n    }c_{\boldsymbol{\lambda}}\chi^{\boldsymbol{\lambda}}.$$
	Then $$\sigma.\pi_{\boldsymbol{\mu}}(g)=\sum_{    \|\boldsymbol{\lambda}\|=n    }c_{\boldsymbol{\lambda}}\sigma.\chi^{\boldsymbol{\lambda}}(g)=\sum_{    \|\boldsymbol{\lambda}\|=n    }c_{\boldsymbol{\lambda}}\chi^{\boldsymbol{\lambda}}(\sigma.g)=\pi_{\boldsymbol{\mu}}(\sigma.g)=\begin{cases}
	1& \text{if~}\sigma.g\in c_{\boldsymbol{\mu}} \\
	0& \text{otherwise.}
	\end{cases}$$ If  $\sigma.g\in c_{\boldsymbol{\mu}}$, then $g\in \sigma^{-1}.c_{\boldsymbol{\mu}}=c_{\sigma^{-1}.\boldsymbol{\mu}}$.
	Therefore, $$\pi_{\boldsymbol{\mu}}(\sigma.g)=\begin{cases}
	1& \text{if~}g\in c_{\sigma^{-1}.\boldsymbol{\mu}} \\
	0& \text{otherwise,}
	\end{cases}$$ which is the same as $\pi_{\sigma^{-1}.\boldsymbol{\mu}}.$ We can conclude that $$\sigma.\pi_{\boldsymbol{\mu}}=\pi_{\sigma^{-1}.\boldsymbol{\mu}}.$$
\end{proof}


Let $\| \boldsymbol{\lambda} \|=n$. Define $$\mathsf{O}_d(\boldsymbol{\lambda})=\{ \sigma.\boldsymbol{\lambda}: \sigma\in \mathsf{Gal}{(|GL_n|,d)}   \}.$$

We denote by $p_r(\varphi)$ the power-sum of a set of variables $Y_{i,\varphi}$, and we let $S_{\lambda}(\varphi)$ be the Schur function of the set of variables $Y_{i,\varphi}$.
Note that  $ch(\chi^{\boldsymbol{\lambda}})=S_{\boldsymbol{\lambda}}=\prod_{\varphi\in \Theta}S_{\boldsymbol{\lambda}(\varphi)}(\varphi)$. For more details see \cite[Chapter 4, Section 4]{MC}.

%
%
%

\begin{proposition}\label{action}
	Let $\boldsymbol{\lambda}$, $\boldsymbol{\gamma}:\Theta\to\mathscr{P}$ with $\lVert{\boldsymbol{\lambda}}\rVert=\lVert{\boldsymbol{\gamma}}\rVert=n$. Then for every $\sigma\in \mathsf{Gal}(|GL_n|,d)$, $$\sigma.\chi^{\boldsymbol{\lambda}}=\chi^{\sigma.\boldsymbol{\lambda}}.$$
\end{proposition}

\begin{proof}

We have that $ch(\chi^{\sigma.\boldsymbol{\lambda}})=S_{{\sigma.\boldsymbol{\lambda}}}$. If we show that $ch(\sigma.\chi^{\boldsymbol{\lambda}})=S_{\sigma.\boldsymbol{\lambda}}$, we must have $\sigma.\chi^{\boldsymbol{\lambda}}=\chi^{\sigma.\boldsymbol{\lambda}}$.
	
	Note that $S_{\boldsymbol{\lambda}(\varphi)}(\varphi)=\sum_{\mu\vdash |\boldsymbol{\lambda}(\varphi)|}c_\mu(\varphi) p_{\mu_1}(\varphi)\ldots p_{\mu_{l(\mu)}}(\varphi)$, for some coefficients $c_\mu(\varphi)$, and so $$S_{\boldsymbol{\lambda}}=\prod_{\varphi\in \Theta}S_{\boldsymbol{\lambda}(\varphi)}(\varphi)=\prod_{\varphi\in \Theta}\sum_{\mu\vdash |\boldsymbol{\lambda}(\varphi)|}c_\mu(\varphi) p_{\mu_1}(\varphi)\ldots p_{\mu_{l(\mu)}}(\varphi).$$
	In order to prove that $ch(\sigma.\chi^{\boldsymbol{\lambda}})=S_{\sigma.\boldsymbol{\lambda}}$, we only need to show that 
	$$\sigma.\left(  ch^{-1}(p_i(\varphi))\right)=ch^{-1}(p_i(\sigma.\varphi))$$ for every positive integer $i$ and every $\varphi\in \Theta$.

Define ${(\lambda,f)}:\Phi \rightarrow  \mathscr{P}$ such that  it  maps $f$ to  $\lambda$ and others to empty partition. Consider that $ch^{-1}(p_i(\varphi))$ can be written as a linear  combination of characteristic function of conjugacy classes, and 
by \cite[Section IV]{MC} we have that
	$$(-1)^{i|\varphi|-1} \sum_{x\in k_{i|\varphi|}^\times}\xi(x) \sum_{|\lambda|=\frac{i|\varphi|}{{\rm deg}(f_x)}}\prod_{j=1}^{\frac{i|\varphi|}{{\rm deg}(f_x)}} (1-q_{f_x}^{-j}) \pi_{(\lambda, f_{x})}  =ch^{-1}(p_i(\varphi)),$$ 
	where
	$f_x$ is the minimal polynomial of $x$ over the $\mathbb{F}_q$ and $\xi$ is an element of $\varphi$. Note that  ${\rm deg}(f_x)={\rm deg}(f_{\sigma^{-1}.x})$ and  $|\sigma.\varphi|=|\varphi|$. Also,
	$$\sigma.\pi_{(\lambda, f_{x})}=\pi_{(\lambda, \sigma^{-1}.f_{x})}=\pi_{(\lambda, f_{\sigma^{-1}.x})}. 	$$
	Therefore,
	$$\sigma.ch^{-1}(p_i(\varphi))=\sigma. \left((-1)^{i|\varphi|-1} \sum_{x\in k_{i|\varphi|}^\times}\xi(x) \sum_{|\lambda|=\frac{i|\varphi|}{{\rm deg}(f_x)}}\prod_{j=1}^{\frac{i|\varphi|}{{\rm deg}(f_x)}} (1-q_{f_x}^{-j}) \pi_{(\lambda, f_{x})} \right)=$$
	$$
	(-1)^{i|\varphi|-1} \sum_{x\in k_{i|\varphi|}^\times}\xi(x) \sum_{|\lambda|=\frac{i|\varphi|}{{\rm deg}(f_x)}}\prod_{j=1}^{\frac{i|\varphi|}{{\rm deg}(f_x)}} (1-q_{f_x}^{-j}) \sigma.\pi_{(\lambda, f_{x})} =
	$$
	$$
	(-1)^{i|\varphi|-1} \sum_{x\in k_{i|\varphi|}^\times}\xi(x) \sum_{|\lambda|=\frac{i|\varphi|}{{\rm deg}(f_x)}}\prod_{j=1}^{\frac{i|\varphi|}{{\rm deg}(f_x)}} (1-q_{f_x}^{-j}) \pi_{(\lambda, f_{\sigma^{-1}.x})} =
	$$
	$$
	(-1)^{i|\varphi|-1} \sum_{\sigma.y\in k_{i|\varphi|}^\times}\xi(\sigma.y) \sum_{|\lambda|=\frac{i|\varphi|}{{{\rm deg}}(f_y)}}\prod_{j=1}^{\frac{i|\varphi|}{{{\rm deg}}(f_y)}} (1-q_{f_y}^{-j}) \pi_{(\lambda, \sigma^{-1}.f_{y})} =
	$$
	$$
	(-1)^{i|\sigma.\varphi|-1} \sum_{y\in k_{i|\sigma.\varphi|}^\times}\sigma.\xi(y) \sum_{|\lambda|=\frac{i|\sigma.\varphi|}{{\rm deg}(f_y)}}\prod_{j=1}^{\frac{i|\sigma.\varphi|}{{\rm deg}(f_y)}} (1-q_{f_y}^{-j}) \pi_{(\lambda, \sigma^{-1}.f_{y})} =ch^{-1}(p_i(\sigma.\varphi)).
	$$

\end{proof}

\begin{theorem}\label{Galois character}
	Let $\boldsymbol{\lambda}$, $\boldsymbol{\gamma}:\Theta\to\mathscr{P}$ with $\lVert{\boldsymbol{\lambda}}\rVert=\lVert{\boldsymbol{\gamma}}\rVert=n$. Then the irreducible characters $\chi^{\boldsymbol{\lambda}}$ and $\chi^{\boldsymbol{\gamma}}$ are constituent of the same $d$-Galois irreducible character if and only if   $$\mathsf{O}_d(\boldsymbol{\lambda})=\mathsf{O}_d(\boldsymbol{\gamma}).$$
\end{theorem}

\begin{proof}
	Assume that $\chi^{\boldsymbol{\lambda}}$ and $\chi^{\boldsymbol{\gamma}}$ are constituent of the same $d$-Galois irreducible character, then there exists $\sigma\in \mathsf{Gal}{(|GL_n|,d)}$ such that $\sigma.\chi^{\boldsymbol{\lambda}}=\chi^{\boldsymbol{\gamma}}$. By Lemma \ref{action}, $\sigma.\chi^{\boldsymbol{\lambda}}=\chi^{\sigma.\boldsymbol{\lambda}}$. Consequently, $\chi^{\boldsymbol{\gamma}}=\chi^{\sigma.\boldsymbol{\lambda}}$	
	 and so $\sigma. \boldsymbol{\lambda}=\boldsymbol{\gamma}$. Therefore,
	$$\mathsf{O}_d(\boldsymbol{\lambda})=\mathsf{O}_d(\boldsymbol{\gamma}).$$

	For the converse, assume that $$\mathsf{O}_d(\boldsymbol{\lambda})=\mathsf{O}_d(\boldsymbol{\gamma}),$$ so that  there is some $\sigma\in \mathsf{Gal}(|GL_n|,d)$ such that 
$\sigma.\boldsymbol{\lambda}=\boldsymbol{\gamma}$, then $S_{\sigma.\boldsymbol{\lambda}}=S_{\boldsymbol{\gamma}}$. Then,
$$\chi^{\boldsymbol{\gamma}}=ch^{-1}(S_{\boldsymbol{\gamma}})=ch^{-1}(S_{\sigma.\boldsymbol{\lambda}})=\chi^{\sigma.\boldsymbol{\lambda}}.$$
 and so $\chi^{\boldsymbol{\gamma}}=\chi^{\sigma.\boldsymbol{\lambda}}$. By Proposition \ref{action}, $\chi^{\sigma.\boldsymbol{\lambda}}=\sigma.\chi^{\boldsymbol{\lambda}}$, and so  $\chi^{\boldsymbol{\gamma}}=\sigma.\chi^{\boldsymbol{\lambda}}$. We can conclude that  $\chi^{\boldsymbol{\lambda}}$ and $\chi^{\boldsymbol{\gamma}}$ are constituent of the same $d$-Galois irreducible character.
\end{proof}

\subsection{The image of Galois  characters by characteristic map}

We define $$S_{ \mathsf{O}_d(\boldsymbol{\lambda} ) }=\sum_{ \boldsymbol{\gamma} \in \mathsf{O}_d(\boldsymbol{\lambda} ) } S_{\boldsymbol{\gamma}},$$ and
$$\chi^{ \mathsf{O}_d(\boldsymbol{\lambda} )  }=\sum_{ \boldsymbol{\gamma} \in \mathsf{O}_d(\boldsymbol{\lambda} ) } \chi^{\boldsymbol{\gamma}}.$$
Therefore, $$\mathsf{GIrr}_d(GL_n)=\{\chi^{\mathsf{O}_d(\boldsymbol{\lambda} ) }: \| {\boldsymbol{\lambda}} \|=n  \}.$$
Let $$\mathsf{C}_d=\mathbb{C}\text{-span}\{S_{ \mathsf{O}_d(\boldsymbol{\lambda} ) }: \boldsymbol{\lambda}  \text{~is a function from~}\Theta \text{~to~}\mathscr{P}  \}.$$

\begin{theorem}
	The map $ch_d: \mathsf{Gcf}_d(GL_*)\rightarrow \mathsf{C}_d$ defined by  $$ch(\chi^{\mathsf{O}_d(\boldsymbol{\lambda} ) })=S_{\mathsf{O}_d(\boldsymbol{\lambda} )}$$ is a Hopf isomorphism.
\end{theorem}

\section{Galois Cuspidal Characters}

It is known that when we consider $\mathsf{cf}(GL_*)$ as a PSH with PSH-basis $\Sigma=\{\chi^{ \boldsymbol{\lambda} }: \boldsymbol{\lambda}  \text{~is a function from~}\Theta \text{~to~}\mathscr{P}  \}$ (see \cite[Chapter 4]{GR} and \cite[Chapter II]{ZE}), $$\mathcal{C}=\Sigma\cap \mathfrak{p}=\{ \chi^{(\square,\varphi)}: \varphi\in \Theta \},$$ where $(\square,\varphi)$ is a function from $\Theta$ to $\mathscr{P}$ which maps $\varphi$ to $(1)$ and others to empty partition. Let $\Sigma_d$ be the PSH-basis of $\mathsf{Gcf}_d(GL_*)$. Any element of  $\mathcal{C}$ is called a cuspidal character. We call an element of 
$$\mathcal{C}_d=\Sigma_d\cap \mathfrak{p}$$ a $d$-Galois cuspidal character. We warn that the elements of $\mathcal{C}_d$ are not necessarily characters.

\begin{corollary}\label{Gsucpidal}
	Let $\sigma\in \mathsf{Gal}{(|GL_n|,d)}$, then $$\sigma.\chi^{(\square,\varphi)}=\chi^{(\square,\sigma.\varphi)}.$$ Moreover, $$\mathcal{C}_{d}=\left\lbrace \left( 1/\sqrt{\mathsf{O}_{d}(\varphi)} \right)  \sum_{\vartheta \in \mathsf{O}_{d}(\varphi)}  \chi^{(\square,\vartheta)}  : \varphi\in \Theta  \right\rbrace.$$
\end{corollary}

\begin{proof}
	If $\varphi\in \Theta$ with $|\varphi|=n$ and $\sigma\in \mathsf{Gal}(|GL_n|,d)$, then by Proposition \ref{pro-copro}, $$\Delta(\sigma.\chi^{(\square,\varphi)})=\sigma.\Delta(\chi^{(\square,\varphi)})=\sigma.(1\otimes \chi^{(\square,\varphi)}+\chi^{(\square,\varphi)}\otimes 1)=$$$$1\otimes \sigma.\chi^{(\square,\varphi)}+\sigma.\chi^{(\square,\varphi)}\otimes 1=1\otimes \chi^{(\square,\sigma.\varphi)}+\chi^{(\square,\sigma.\varphi)}\otimes 1.$$ Therefore, 
	$$\Delta\left(\left( 1/\sqrt{\mathsf{O}_{d}(\varphi)} \right)  \sum_{\vartheta \in \mathsf{O}_{d}(\varphi)}  \chi^{(\square,\vartheta)} \right)=$$
	$$\left(  1/\sqrt{\mathsf{O}_{d}(\varphi)} \right)  \sum_{\vartheta \in \mathsf{O}_{d}(\varphi)}  1\otimes \chi^{(\square,\vartheta)} +\left( 1/\sqrt{\mathsf{O}_{d}(\varphi)} \right)  \sum_{\vartheta \in \mathsf{O}_{d}(\varphi)}  \chi^{(\square,\vartheta)} \otimes 1,$$ and so $$\left( 1/\sqrt{\mathsf{O}_{d}(\varphi)} \right)  \sum_{\vartheta \in \mathsf{O}_{d}(\varphi)}  \chi^{(\square,\vartheta)} \in \Sigma_d\cap\mathfrak{p}.$$
	Now we want to show that if $ \left( 1/\sqrt{\mathsf{O}_{d}(\varphi)} \right)  \sum_{\phi\in \mathsf{O}_{d}(\psi)}  \phi$ is in $\mathcal{C}_d$,	then $\psi\in \mathcal{C}$. 
	Consider that 
	$$ \left( 1/\sqrt{\mathsf{O}_{d}(\varphi)} \right)  \sum_{\phi\in \mathsf{O}_{d}(\psi)}  \phi=\left( 1/\sqrt{\mathsf{O}_{d}(\varphi)} \right) \sum_{\sigma\in A}\sigma.\psi$$ for some $A\subseteq \mathsf{Gal}(|GL_n|,d)$.

	$$\Delta\left( \left( 1/\sqrt{\mathsf{O}_{d}(\varphi)} \right) \sum_{\sigma\in A}\sigma.\psi\right) = \left(1/\sqrt{\mathsf{O}_{d}(\varphi)} \right) \sum_{\sigma\in A}\Delta\left( \sigma.\psi\right) $$
	$$=1\otimes   \left( 1/\sqrt{\mathsf{O}_{d}(\varphi)} \right) \sum_{\sigma\in A}\sigma.\psi+ \left( 1/\sqrt{\mathsf{O}_{d}(\varphi)} \right) \sum_{\sigma\in A}\sigma.\psi \otimes 1.$$
	Note that for every $\sigma\in A$, $\Delta(\sigma.\chi)$ is a non-negative linear combination of $d$-Galois irreducible characters. Therefore, it is not possible to have the equation above if we do not have   $\Delta(\sigma.\psi)=\sigma.\psi\otimes 1+1\otimes \sigma.\psi$ for every $\sigma\in A$. Therefore, $\psi\in \mathcal{C}$.
\end{proof}

\begin{remark}
	A result similar to above corollary can be shown for any PSH in Theorem \ref{PSH}. 
\end{remark}
The following corollary follows from the discussion above and Section \ref{SPSH}.

\begin{corollary}
Let $\mathcal{C}_d$ be as in Corollary \ref{Gsucpidal}. Then   $$\mathsf{Gcf}_d(GL_*)\cong \bigotimes_{\rho\in \mathcal{C}_d} \mathsf{Gcf}_d(GL_*)(\rho).$$
\end{corollary}

\section{Related problems}

We conclude this paper with a list of open problems. Let $p$ be a prime and let $q$ be a prime
power.
\begin{enumerate}
	\item The character theory of the family of unipotent upper-triangular matrices over a
finite field $\mathbb{F}_q$, $UT_n(q)$ is known to be wild. Is there a tame Galois character theory
for the family of the unipotent upper-triangular matrices?
\item  The supercharacter theory defined by Diaconis and Isaacs \cite{DI08} for $UT_n(q)$ is a $p$-Galois supercharacter theory. 
Can we precisely express the supercharacters of the supercharacter
theory by Diaconis and Isaacs as a linear combination of $p$-Galois irreducible
characters?
\item Can we count the number of Galois irreducible characters for the family of unipotent
upper-triangular matrices? Is it a polynomial of $p$ and $n$?

\item  Any normal supercharacter theory is integral \cite{AL}, so any normal subgroup is a
union of $1$-Galois classes, and any supercharacter of a normal supercharacter theory
is a non-negative linear combination of $1$-Galois irreducible characters. Can we precisely
express a supercharacter of a normal supercharacter theory for $UT_n(q)$ as a linear
combination of $1$-Galois irreducible characters?

\end{enumerate}


\begin{thebibliography}{9}
	\bibitem{AABB12}M. Aguiar, C. Andr\'e, C. Benedetti, N. Bergeron, Z. Chen, P. Diaconis, A. Hendrickson, S.
	Hsiao, I. M. Isaacs, A. Jedwab, K. Johnson, G. Karaali, A. Lauve, T. Le, S. Lewis, H. Li, K.
	Magaard, E. Marberg, J-C. Novelli, A. Pang, F. Saliola, L. Tevlin, J-Y. Thibon, N. Thiem, V.
	Venkateswaran, C. R. Vinroot, N. Yan and M. Zabrocki, Supercharacters, symmetric functions in
	noncommuting variables, and related Hopf algebras, Adv. Math. 229 (2012) 2310-2337.
	\bibitem{An} S. Andrews, The Hopf monoid on nonnesting supercharacters of pattern groups, Journal of Algebraic Combinatorics, 42 (2015) 129–164.
	\bibitem{ABT} Marcelo Aguiar, Nantel Bergeron, Nat Thiem, Hopf monoids from class functions on unitriangular matrices,
	Algebra and Number Theory (2013) 1743-1779.
	\bibitem{AL} Farid Aliniaeifard, Normal supercharacter theories and their supercharacters,
	Journal of Algebra (2017) 469, 464-484.
	\bibitem{BT}	N. Bergeron and Nat Thiem, A supercharacter table decomposition via power-sum symmetric functions, Internat. J. Algebra Comput. 23 (2013) 763-778.
	\bibitem{DI08} P. Diaconis, M. Isaacs, Supercharacters and superclasses for algebra groups, Trans. Amer. Math. Soc. 360 (2008) 2359-2392.
	\bibitem{IS} I. M. Isaacs, Character Theory of Finite Groups, Academic Press, New York, 1976.
	\bibitem{Ge} L. Geissinger, Hopf algebras of symmetric functions and class functions, Lecture Notes in Math. 579 (1977)
	168–181.
	\bibitem{Ke} Justin Keller, Generalized Supercharacter Theories and Schur Rings for Hopf Algebras, Ph.D Thesis, University of Colorado Boulder, 2014. 
	\bibitem{GR} D. Grinberg, V. Reiner, Hopf algebras in Combinatorics, arXiv preprint arXiv:1409.8356, 2018.
	\bibitem{MC} I.G. Macdonald, Symmetric functions and Hall polynomials, 2nd ed., Oxford University Press, 1995.
	\bibitem{ZE} A.V. Zelevinsky. Representations of finite classical groups: a Hopf algebra approach. Lecture Notes in Mathematics 869.
	Springer-Verlag, Berlin-New York, 1981.
\end{thebibliography}
\end{document}